\documentclass[11pt, a4paper, reqno]{amsart}
\usepackage{amsmath}
\usepackage{amsfonts}
\usepackage{amssymb}
\usepackage{amsthm}
\usepackage{url}
\usepackage{color}
\usepackage{array}

\newcommand{\R}{\mathbb{R}}

\newcommand{\N}{\mathbb{N}}

\DeclareMathOperator*{\dist}{dist} 
\DeclareMathOperator*{\pipdist}{\text{dist}_\pip}

\def\vint_#1{\mathchoice%
          {\mathop{\kern 0.2em\vrule width 0.6em height 0.69678ex depth -0.58065ex
                  \kern -0.8em \intop}\nolimits_{\kern -0.4em#1}}%
          {\mathop{\kern 0.1em\vrule width 0.5em height 0.69678ex depth -0.60387ex
                  \kern -0.6em \intop}\nolimits_{#1}}%
          {\mathop{\kern 0.1em\vrule width 0.5em height 0.69678ex depth -0.60387ex
                  \kern -0.6em \intop}\nolimits_{#1}}%
          {\mathop{\kern 0.1em\vrule width 0.5em height 0.69678ex depth -0.60387ex
                  \kern -0.6em \intop}\nolimits_{#1}}}

\newcommand{\Om}{\Omega}

\newcommand{\pip}{\varphi}

\newcommand{\ch}{\text{\raise 1.3pt \hbox{$\chi$}\kern-0.2pt}}
\newcommand{\dOmp}{d_{\Om_\pip}}

\theoremstyle{plain}
\newtheorem{theorem}[equation]{Theorem}
\newtheorem{lemma}[equation]{Lemma}

\numberwithin{equation}{section}

\theoremstyle{definition}
\newtheorem{definition}[equation]{Definition}

\theoremstyle{remark}
\newtheorem{remark}[equation]{Remark}

\begin{document}

\title[Transformation of uniform domains via distance weights]{Conformal transformation of uniform domains 
under weights that depend on distance to the boundary}

\author[Gibara, Shanmugalingam]{Ryan Gibara, Nageswari Shanmugalingam}
\thanks{N.S. is partially supported by
the NSF (U.S.A.) grant DMS~\#2054960. Part of the work on this paper was done
while N.S. was visiting MSRI in Spring 2022 to participate in a program supported
by the NSF (U.S.A.) grant DMS~\#1928930. She wishes to thank MSRI for its kind
hospitality. We also thank the kind referee for valuable comments pointing out 
inaccuracies in early manuscripts of the paper. Authors state no conflict of interest.
}

\date{\today}


%\maketitle

\begin{abstract}
The sphericalization procedure converts a Euclidean space into a compact sphere.
In this note we propose a variant of this procedure for locally compact, rectifiably path-connected,
non-complete, unbounded
metric spaces by using conformal deformations that depend only on the distance to the boundary of the metric space.
This deformation is locally bi-Lipschitz to the original domain 
near its boundary, but transforms the space
into a bounded domain. We will show that if the original metric space is a uniform domain with respect to its completion,
then the transformed space is also a uniform domain.
\end{abstract}

%\noindent
%    {\small Mathematics Subject Classification (2020): Primary: 30L05; Secondary: 30L10.}

\maketitle

\section{Introduction}

The stereographic projection identifies the one-point compactified complex plane with the unit sphere, and this
identity has been exploited in the study of analytic functions and conformal maps between planar regions and their
behavior at infinity. Higher
dimensional stereographic projections also identify the one-point compactification of $\R^n$ with the 
$n$-dimensional unit sphere in $\R^{n+1}$. The work of~\cite{BK, BHX} formulated
a fruitful generalization of stereographic projection to more general metric spaces, and this formulation has been
used in the literature to study Gromov hyperbolic spaces, uniform domains, and quasi-M\"obius maps, see for 
example~\cite{BBL, BK, BHX, DL1, DL2, HSX, L, LS, ZLL}. 
A metric space \emph{inversion}
about a point $p$ in a metric space $X$ turns a bounded metric space $X\setminus\{p\}$ into an unbounded
metric space, and the \emph{sphericalization} of $X$ with respect to the base point $p\in X$ turns
an unbounded metric space $X$ into a bounded space whose completion is topologically the one-point compactification of $X$.
It was also shown in~\cite{BHX} that the sphericalization and inversion operations of uniform domains yield 
uniform domains. Moreover, they show that if the metric space satisfies a geometric condition called 
annular quasiconvexity, then there is control over how the uniformity constant is transformed by these
operations.

Apart from the study of quasiconformal geometry (see~\cite{BHK,MartSarv} for a tiny sampling of offerings from this line
of enquiry), 
uniform domains also play a key role in potential theory as
uniform domains in a complete doubling metric measure space supporting a $p$-Poincar\'e inequality are also
known to support a $p$-Poincar\'e inequality~\cite{BS} and admit a description of traces of Sobolev-class
functions on the domain as belonging to certain Besov classes~\cite{Maly}. Hence, much of the
potential theory and Dirichlet problems on smooth domains are extendable to uniform domains. 
An additional sampling of the vast literature on potential theory related to uniform domains can be found in~\cite{Azz, Lierl}.
For these reasons,
a systematic study of metric transformations that preserve the uniform domain property is desirable, and the
goal of the present note is to contribute to this study.

While sphericalization and inversion are analogues of stereographic projection and its inverse, these operators
distort the metric on $X$ everywhere, including near the boundary of $X$ if $X$ is not complete. In certain
circumstances we would wish not to distort the metric, at least locally, near the boundary of $X$; for example,
if $X$ is a  uniform domain (and hence is locally compact but non-complete), we would wish to preserve the
local nature of the metric on $X$ near $\partial X:=\overline{X}\setminus X$, where 
$\overline{X}$ denotes the completion of $X$, while transforming $X$ into a bounded
space. Such transformation is desirable if we are interested in studying boundary-value problems on $X$ in 
terms of boundary value problems on bounded domains (see for example~\cite{CKKSS}) and so find information
about growth-at-infinity behavior of solutions in the unbounded domain. The purpose of this
note is to propose a range of modifications of the sphericalization procedure of~\cite{BHX} so that the modification
does not perturb the inner length metric, locally, near $\partial X$. Here, by the inner length metric, we mean the metric
$d_{\rm inn}$ given by $d_{\rm inn}(x,y)=\inf_\gamma\ell_d(\gamma)$, where the infimum is over all curves in $\Om$
with end points $x$ and $y$.

To this end, we consider $(\Om, d)$ to be a locally compact, non-complete metric space such that $\Om$ is a 
uniform domain in $\overline{\Om}$ (that is, $\Om$ is a uniform
space in the language of~\cite{BHK}). We also fix a monotone decreasing continuous function $\pip:(0,\infty)\to(0,\infty)$
such that $\pip(t)=1$ when $0<t\le 1$, 
\[
\int_0^\infty \pip(t)\, dt<\infty,
\]
and there is a constant $C_\pip\ge 1$ such that we have $\pip(t)\le C_\pip\,\pip(2t)$ for all $t>0$. As $\Om$ is a uniform domain,
it is rectifiably connected, that is, pairs of points in $\Om$ can be connected by curves in $\Om$ of finite length. Hence,
we use $\pip$ to construct a new metric $d_\pip$ on $\Om$ by setting 
\[
d_\pip(x,y):=\inf_\gamma\ \int_\gamma\pip\circ d_\Om\, ds,
\]
with the infimum ranging over all rectifiable curves in $\Om$ with end points $x,y\in\Om$. 
Here, $\int_\gamma h\, ds:=\int_\gamma h(\gamma(\cdot))\,ds$ is the path integral with respect 
to the arc-length parametrization of the
rectifiable curve $\gamma$, see for example~\cite[Chapter 5]{HKST}.
In the above, $d_\Om$ is defined by $d_\Om(x)=\dist(x,\partial\Om)$, see Definition~\ref{def:dom} below.

We will show in Section~2 that $d_\pip$ and $d$ are locally bi-Lipschitz
near $\partial\Om$ and that the completion of $\overline{\Om}$ with respect to $d_\pip$ is topologically a one-point
compactification of $\overline{\Om}$. We denote $\Om_\pip:=\overline{\Om\cup\partial\Om}^\pip\setminus\partial\Om$,
where $\overline{A}^\pip$ is the completion of $A\subset\overline{\Om}$ with respect to the metric $d_\pip$. 
The following is the main theorem of this note.

\begin{theorem}\label{thm:main}
The domain $\Om_\pip$, equipped with the metric $d_\pip$, is a uniform domain with
$\partial\Om_\pip=\partial\Om$ and uniformity constant depending only on
the constant $C_\pip$ and the uniformity constant $C_U$ associated with the metric $d$. 
Furthermore, the natural identity map $I:\Om\to\Om_\pip\setminus\{\infty\}$ is a
local bi-Lipschitz map, and it is also uniformly locally bi-Lipschitz near $\partial\Om$. If $\Om$ is a length space, then
$I$ is a local isometry near $\partial\Om$.
\end{theorem}

One additional advantage of the above is that we do not need annular quasiconvexity of $\Om$ with respect to the metric $d$
in order to gain quantitative control over the uniformity constant with respect to the metric $d_\pip$. 

We end the discussion in this section by describing a simple illustrative example. Recall that the upper half-plane
$\Om:=\R\times(0,\infty)$ is a uniform domain and that the sphericalization procedure on $\Om$ gives an isometric
copy of a spherical cap in $\mathbb{S}^2$. Fixing $\beta>1$,
with the choice of $\pip(t)=t^{-\beta}$ for $t>1$ and $\pip(t)=1$ for $0<t\le 1$, we have that \emph{as a set},
$\Om_\pip=(\R\times(0,\infty))\cup\{\infty\}$. Note that the Euclidean boundary $\partial\Om=\R\times\{0\}$.
For \emph{each} $z\in\R\times\{0\}$, a calculation shows that $d_\pip(z,\infty)=\tfrac{\beta}{\beta-1}$, and so
$\infty$ is \emph{not} an accumulation point (with respect to the metric $d_\pip$) of $\R\times\{0\}$;
this is in contrast to the sphericalization of $\Om$, which has an accumulation point of $\R\times\{0\}$. Note also
that for each $z\in\R\times\{0\}$ we have that $d_\pip(z,z+1)=1$, and hence, in this
example,  $\overline{\Om_\pip}^\pip$ is not compact.

The above example illustrates the fact that given $\Om$ and $\pip$ as in this note, there is a positive distance
(with respect to the metric $d_\pip$) between $\infty$ and $\partial\Om_\pip=\partial\Om$; see 
Lemma~\ref{lem:d-bdry-to-phi} and its proof below.

\medskip

\section{Preliminaries}

Consider an unbounded metric space $(\Omega, d)$ with completion $\overline\Omega$ and boundary 
$\partial\Omega:=\overline\Omega\setminus\Omega$. We say that $\Om$ is a \emph{uniform domain} if it 
is locally compact and non-complete, and there exists a constant $C_U\ge 1$ for which each pair 
$x,y\in\Omega$ with $x\neq y$ can be connected by a $C_U$\emph{-uniform curve} $\gamma$. That is, $\gamma$ satisfies the following:
\begin{itemize}
	\item its length (with respect to the metric $d$) satisfies $\ell_d(\gamma)\le C_U\, d(x,y)$;
	\item for each $z$ in the trajectory of $\gamma$, we have
	\[
	\min\{\ell_d(\gamma_{x,z}),\ell_d(\gamma_{z,y})\}\le C_U\, \dist(z,\partial\Om).
	\]
\end{itemize}

By increasing the value of $C_U$ if need be, we can assume that subcurves of uniform curves are also uniform, see~\cite{BHK}.
Moreover, for each point $\zeta\in\partial\Om$ and $x\in\Om$ we can find a $C_U$-uniform curve $\gamma:[0,L)\to\Om$
such that $\gamma(0)=x$ and $\lim_{t\to L^-}\gamma(t)=\zeta$. 

More generally, given $x\in\Om$ and $\zeta\in\partial\Om$, we say that a curve
$\beta:[0,L)\to\Om$ has end points $x$ and $\zeta$ with respect to the metric $d$ if $\beta(0)=x$ and
$\lim_{t\to L^-}\beta(t)=\zeta$, the limit being taken with respect to $d$.

\begin{definition}\label{def:dom}
For $x\in\Omega$, we set $d_\Omega(x):=\dist(x,\partial\Omega)$. 
We let
\[
\Om_0:=\{x\in\Om\, :\, d_\Om(x)\le 1\},
\]
and, for positive integers $n$, we set
\[
\Om_n:=\{x\in\Om\, :\, 2^{n-1}<d_\Om(x)\le2^n\}.
\]
Note that $\Om=\bigcup_{n=0}^\infty\Om_n$.
\end{definition}

We fix a monotone decreasing continuous function $\pip:(0,\infty)\to(0,\infty)$ such that
$\pip(t)=1$ when $0<t\le 1$, there is a constant $C_\pip\ge 1$ such that we have $\pip(t)\le C_\pip\, \pip(2t)$ for all $t>0$, and 
\begin{equation*}
\int_0^\infty\pip(t)\, dt<\infty. 
\end{equation*}
This condition ensures, by quasiconvexity of $\Omega$, that the metric space $(\Om_\pip,d_\pip)$ is bounded, see 
Lemma~\ref{lem:dist-to-infty}.
The above condition is equivalent to the condition we will use frequently in this note:
\begin{equation}\label{eq:pip-reverse-doubling}
\sum_{n=0}^\infty 2^n\, \pip(2^n)<\infty.
\end{equation}
The prototype function $\pip$ to keep in mind is $\pip(t)=t^{-\beta}$ for some $\beta>1$ when $t>1$,
or $\pip(t)=t^{-\beta}(1+\log(t))^{-\kappa}$ for some $\beta>1$ and $\kappa>0$ when $t>1$. The first prototype function
is used in~\cite[Section~7]{CKKSS} to convert an unbounded uniform domain $Z\times[0,\infty)$, with $Z$ a compact length space,
into a bounded uniform domain.

{ \bigskip
\noindent {\bf Standing assumptions:} \emph{In summary, we will assume that $(\Omega, d)$ is an
unbounded uniform domain and that subcurves of all uniform curves are also uniform. Moreover, $\pip:(0,\infty)\to(0,\infty)$
is a monotone decreasing function such that $\pip(t)=1$ for $t\le 1$, $\int_0^\infty\pip(t)\, dt$ is finite, and there is a constant
$C_\pip\ge 1$ such that for all $t>0$ we have $\pip(t)\le C_\pip\, \pip(2t)$.}}

We use $\pip$ to, in the language of~\cite{BHK}, dampen the metric $d$ on $\Om$ by modifying it to 
$d_\pip$. We define this new metric by setting
\[
d_\pip(x,y):=\inf_\gamma \int_\gamma \pip\circ d_\Om\, ds=:\inf_\gamma\int_\gamma \pip(d_\Om(\gamma(t)))\, dt,
\]
with the infimum ranging over all rectifiable curves $\gamma$ in $\Omega$ with end points $x$ and $y$; we 
consider the arc-length (with respect to the original metric $d$) parametrization of $\gamma$. As $\pip(t)\le 1$ for all $t>0$,
we have that $d_\pip(x,y)\le  C_U d(x,y)$ whenever $x,y\in\Om$.  %****

Now we have two identities for the set $\Om$; namely, $(\Om,d)$ and $(\Om,d_\pip)$. Since $\pip=1$ on $\Om_0$,
both metrics $d$ and $d_\pip$ extend as metrics to $\Om\cup\partial\Om$. We will show this below in Lemma~\ref{lem:d-bdry-to-phi}. 

We denote $\Om_\pip:=\overline{\Om\cup\partial\Om}^\pip\setminus\partial\Om$,
where $\overline{A}^\pip$ is the completion of $A\subset\overline{\Om}$ with respect to the metric $d_\pip$.
First, we show in the following lemma that there is only one point in the completion of $\overline\Om$ with respect to $d_\pip$
that is not in the completion of $\Om$ with respect to $d$. Denoting this point by $\infty$, it follows 
that $\Omega_\pip=\Omega\cup\{\infty\}$.

\begin{lemma}\label{lem:noncomplete}
There is a sequence in $\Om$ that is Cauchy with respect to the metric $d_\pip$ but not with respect to $d$. Any two 
$d_\pip$-Cauchy sequences that are not $d$-Cauchy sequences must be equivalent with respect to the metric $d_\pip$.
\end{lemma}

\begin{proof}
We fix $x_0\in\Om_2$, and we choose $x_j\in\Om_j$ for each integer $j\ge 3$. Let $\beta_j$ be a $C_U$-uniform curve in
$\Om$ with end points $x_0,x_j$. Since $\Om$ is locally compact, we can exhaust $\Om$ by a sequence of proper
subdomains $D_k$ of $\Om$ such that $x_0\in D_k\Subset D_{k+1}$. Here $D_k\Subset D_{k+1}$ means that the closure of $D_k$ is a compact subset of $D_{k+1}$. By the Arzel\`a-Ascoli theorem, we can then
find a curve $\beta_\infty$ and a subsequence of $\beta_j$, also denoted
$\beta_j$, such that for each $k$ the segments of the curves $\beta_j$ lying in $D_k$ converge uniformly to 
the segment of $\beta_\infty$ in $D_k$. Since each $\beta_j$ is a $C_U$-uniform curve, so is $\beta_\infty$ and each 
of its subcurves.

We now use $\beta_\infty$ to construct a sequence that is Cauchy with respect
to $d_\pip$ but not Cauchy with respect to $d$. 
Since $d(x_0,x_j)\to\infty$, it follows that $\ell_d(\beta_j)\to\infty$ as $j\to\infty$, and so $\ell_d(\beta_\infty)=\infty$. Hence,
by the uniformity of $\beta_\infty$, we know that for each $j\ge 3$ the curve $\beta_\infty$ intersects $\Om_j$.
For each $j\ge 3$, we set $y_j=\beta_\infty(t_j)$ to be the first time $\beta_\infty$ intersects
$\Om_j$. Then for each $m\in\N$, $d(y_j,y_{j+m})\geq 2^{j+m-1}-2^j$, showing that $(y_j)_j$ is not Cauchy with respect to $d$. Let $\gamma_j$ be a subcurve of $\beta_\infty$ with end points $y_j$ and $y_{j+1}$. Then, as $\pip$ is monotone 
decreasing, we have
\begin{align*}
d_\pip(y_j,y_{j+1})\le \ell_\pip(\gamma_j)=\int_{\gamma_j}\pip(d_\Om(\gamma_j(t)))\, dt
  &\le \pip(2^j/C_U)\ell_d(\gamma_j)\\
  &\le \pip(2^j/C_U)\, C_Ud_\Om(y_{j+1})\\
  &\le 2C_U\, 2^j\pip(2^j/C_U).
\end{align*}
By the reverse doubling property of $\pip$, and from the above inequality it follows that
there is a positive constant $C$ depending solely on $C_U$ and
$C_\pip$ such that 
\[
d_\pip(y_j,y_{j+1})\le C\, 2^j\pip(2^j).
\]
It follows that for each $j\ge 3$ and $m\in\N$,
\[
d_\pip(y_j,y_{j+m})\le C\, \sum_{n=j}^{j+m-1}2^n\pip(2^n).
\]
The above inequality and~\eqref{eq:pip-reverse-doubling} guarantee that the sequence
$(y_j)_j$ is Cauchy with respect to $d_\pip$.

Let $(z_j)_j$ be another sequence in $\Om$
that is Cauchy with respect to $d_\pip$ but not with respect to $d$. Then we know that 
$\lim_jd_\Om(z_j)=\infty$. Indeed, if there is some $k_0\in\N$ such that each $z_j\in\bigcup_{n=0}^{k_0}\Om_n$, then
for $j,k\in\N$ and $\gamma$ any curve connecting $z_j$ and $z_k$ with $\ell_\pip(\gamma)\le \tfrac{11}{10}d_\pip(z_j,z_k)$ 
we have either that
$\gamma$ lies in $\bigcup_{n=0}^{k_0+1}\Om_n$, in which case
\[
\ell_\pip(\gamma)=\int_\gamma\pip(d_\Om(\gamma(t)))\, dt\ge \pip(2^{k_0+1})\ell_d(\gamma)\ge \pip(2^{k_0+1})d(z_j,z_k),
\]
or else that $\gamma$ intersects $\Om_{k_0+2}$, in which case
\[
\ell_\pip(\gamma)\ge \int_{\gamma\cap\Om_{k_0+1}}\pip(d_\Om(\gamma(t)))\, dt \ge \pip(2^{k_0+1})\ell_d(\gamma\cap\Om_{k_0+1})\ge 2^{k_0}\pip(2^{k_0+1}).
\]
In both of these cases, we get that $\limsup_{j,k\to\infty}d_\pip(z_j,z_k)>0$, violating the $d_\pip$-Cauchy property.

To complete the proof of the lemma, we show that $(z_j)_j$ is $d_\pip$-equivalent to $(y_j)_j$, that is,
$\lim_jd_\pip(z_j,y_j)=0$. Since $\lim_jd_\Om(z_j)=\infty$ and $\lim_jd_\Om(y_j)=\infty$, for each integer $m\ge 2$ 
both $z_j$ and $y_j$ belong to $\bigcup_{n=m}^\infty\Om_n$ for sufficiently large $j$. So, for sufficiently large $j$,
with $\alpha_j$ a $C_U$-uniform curve in $\Om$ with end points $z_j$ and $y_j$, let $\alpha_j:[0,L]\to\Om$ be the standard
arclength parametrization of $\alpha_j$ (with respect to the metric $d$). Then with $n_0$ the positive integer for which
$2^{n_0-1}<C_U\le 2^{n_0}$,
\begin{align*}
\ell_\pip(\alpha_j\vert_{[0,L/2]})=\sum_{n=m-n_0}^\infty \ell_\pip(\alpha_j\cap\Om_n)
  &\le \sum_{n=m-n_0}^\infty \pip(2^n)\ell_d(\alpha_j\cap\Om_n)\\
  &\le \sum_{n=m-n_0}^\infty\pip(2^n)\, C_U\, 2^n.
\end{align*}
An analogous treatment of $\alpha_j\vert_{[L/2,L]}$ then tells us that
\[
\ell_\pip(\alpha_j)\le 2\,C_U\, \sum_{n=m-n_0}^\infty2^n\, \pip(2^n),
\]
which goes to zero as $m\to\infty$ by~\eqref{eq:pip-reverse-doubling}. 
\end{proof}

\begin{remark}\label{rem:def-n0}
Here, and in the rest of the paper, $n_0$ is the positive integer such that $2^{n_0-1}\le C_U< 2^{n_0}$.
The above proof also yields an additional property, namely, that if $y,z\in\bigcup_{n=m}^\infty\Om_n$ for some $m\ge n_0$,
then  with $\gamma$ a $C_U$-uniform curve with end points $x$ and $y$, we have 
\[
d_\pip(y,z)\le \ell_\pip(\gamma)\le 2\, C_U\, \sum_{n=m-n_0}^\infty2^n\, \pip(2^n).
\]
\end{remark}

\begin{remark}\label{rem:useful1}
 From the fact that $\int_0^\infty \pip(t)\, dt$ is finite, we see that whenever $c>0$, then there is a positive integer $N_c$ such that
 whenever $n\in\N$ with $\int_{2^n}^\infty\pip(t)\, dt\ge c$, we have $n\le N_c$.
\end{remark}

\begin{lemma}\label{lem:d-bdry-to-phi}
$\partial\Om_\pip=\partial\Om$. Moreover, the extensions of the metrics $d$ and $d_\pip$ to $\partial\Om$ and $\partial\Om_\pip$
are locally bi-Lipschitz.  
\end{lemma}

We point out here that this lemma does not require $\Om$ to be a uniform domain,
but we do need $\Om$ to be $C_q$-quasiconvex with respect to the metric $d$; in this case the reader should
replace $C_U$ in the following proof with $C_q$, the quasiconvexity constant. 

\begin{proof}
If $(x_j)_{j}$ is a sequence in $\Omega$ that converges to a point $\zeta\in\partial\Om$
with respect to the metric $d$, then we have that 
this sequence is also Cauchy in $\Om_\pip$. 
Moreover, for sufficiently large $j,k$ we have that
$d(x_j,x_k)<\tfrac{1}{10C_U}$ and 
$d(x_j,\zeta)<\tfrac{1}{10C_U}$. 
It follows that for sufficiently large $j$, $x_j\in\Om_0$, and
if $\gamma$ is a curve in $\Om$ with end points $x_j$ and $x_k$ such that
$\ell_\pip(\gamma)\le \tfrac{11}{10}d_\pip(x_j,x_k)$,
then $\gamma$ cannot leave $\Om_0$; for if it does, then
\[
\tfrac{11}{10}d_\pip(x_j,x_k)\ge \ell_\pip(\gamma)=\int_\gamma \pip(d_\Om(\gamma(t)))\, dt\ge \ell_d(\gamma_0)\ge \tfrac{9}{10},
\]
where $\gamma_0$ is the largest subcurve of $\gamma$ with one end point $x_j$ and such that $\gamma_0\subset\Omega_0$.
Since $d_\pip(x_j,x_k)\le C_U d(x_j,x_k)< \tfrac{1}{10}$, this leads to a contradiction. 

We now show that the sequence $(x_j)_j$, which is Cauchy with respect to both $d$ and $d_\pip$, cannot be 
$d_\pip$-equivalent to any $d_\pip$-Cauchy sequence $(y_j)_j$ that is not a $d$-Cauchy sequence (see 
Lemma~\ref{lem:noncomplete} for the existence and uniqueness of such a sequence $(y_j)_j$, which is 
denoted in this paper by $\infty$). Indeed, as $(y_j)_j$ is not 
Cauchy with respect to $d$, it is not equivalent to $(x_j)_j$ with respect to the metric $d$. Hence there is some
$0<c<1/(10C_U)$ such that for sufficiently large $j$ (perhaps after passing to a subsequence if necessary), 
we have that $d(x_j,y_j)>c$. Now, if $\beta$ is any curve
in $\Om$ connecting $x_j$ to $y_j$, we must then have that $\beta$ starts from $x_j$ and leaves the ball
$B_d(x_j,c)\subset \Om_0$, and so
\[
\ell_\pip(\beta)\ge \int_{\beta\cap B_d(x_j,c/2)} \, ds\ge c/2.
\]
Taking the infimum over all such $\beta$ gives us $d_\pip(x_j,y_j)\ge c/2$.
It follows that $(x_j)_j$ cannot be equivalent to $(y_j)_j$ with respect to the metric $d_\pip$, that is, $(x_j)_j$ cannot
converge to $\infty$ in the metric $d_\pip$.

Moreover, if $(x_j)_{j}$ and $(y_j)_{j}$ are two non-equivalent $d$-Cauchy sequences in $\Om$, converging to two distinct
points $\zeta,\eta\in\partial\Om$, then for sufficiently large $j$ we have that $d(x_j,y_j)\ge \tau=d(\zeta,\eta)/2>0$. In this
case, any curve $\gamma$ connecting $x_j$ to $y_j$ in $\Omega$ must have length $\ell_d(\gamma)\ge \tau$. 
If such $\gamma$ does not stay within $\Om_0$, then an argument as above tells us that $\ell_\pip(\gamma)\ge \tfrac{9}{10}$
when $j$ is large. If $\gamma$ stays entirely within $\Om_0$, then $\ell_\pip(\gamma)=\ell_d(\gamma)\ge \tau$. It follows
that $d_\pip(x_j,y_k)\ge \min\{\tau,\tfrac{9}{10}\}>0$, and thus the two sequences are not Cauchy-equivalent with respect to the
metric $d_\pip$ either. That is, $\partial\Om\subset \partial\Om_\pip$.

Now suppose that $(x_j)_{j}$ is a Cauchy sequence in $\Om_\pip$ that does not converge. Then, in particular, there is some
$r>0$ such that for sufficiently large $j$ we have that $x_j\not\in B_\pip(\infty, r)$. It follows then from 
Remark~\ref{rem:useful1} that there is some $k_0\in\N$ such that when $j$ is sufficiently large, we have 
$x_j\in\bigcup_{n=0}^{k_0}\Om_n$. For such sufficiently large $j,k\in\N$, let $\gamma$ be a curve in $\Om$ with end points
$x_j,x_k$ such that $d_\pip(x_j,x_k)\le \tfrac{11}{10}\ell_\pip(\gamma)$, then we consider two cases. Either $\gamma$ is
entirely inside $\bigcup_{n=0}^{2k_0}\Om_n$, in which case we have
\[
\ell_\pip(\gamma)\ge \pip(2^{2k_0})\, \ell_d(\gamma)\ge \pip(2^{2k_0})d(x_j,x_k),
\]
or else $\gamma$ intersects $\Omega_{2k_0+1}$, in which case, we have that
\[
\ell_\pip(\gamma)
   \ge \pip(2^{2k_0})\ell_d(\gamma\cap\Om_{2k_0})\ge 2^{2k_0-1}\pip(2^{2k_0}).
\]
This latter case is not possible for sufficiently large $j$ and $k$, since by choice, 
$\lim_{j,k\to\infty} d_\pip(x_j,x_k)=0$. The former case is therefore
the only possibility for sufficiently large $j,k$, and hence $(x_j)_{j}$ is Cauchy with respect to the original metric $d$.
As this sequence does not converge with respect to $d_\pip$, it follows that it does not converge in $\Om$ with respect to $d$ 
either; hence, $\partial\Om_\pip\subset\partial\Om$.

The above argument also shows that if $\zeta,\eta\in\partial\Om$ with $d(\zeta,\eta)\le \tfrac{1}{10}$, then
$d(\zeta,\eta)\le d_\pip(\zeta,\eta)\le C_Ud(\zeta,\eta)$, where we used the quasiconvexity of $\Om$ with 
respect to the metric $d$. Thus, the two metrics are locally bi-Lipschitz.  
\end{proof}

For $x\in\Om_\pip$, we set
\[
d_{\Om_\pip}(x):=\pipdist(x,\partial\Om_\pip):=\inf\{d_\pip(x,\zeta)\, :\, \zeta\in\partial\Om_\pip\}. 
\]
We now consider some preliminary lemmas that will be useful in showing that $\Omega_\pip$ is uniform. 

When $m$ is a non-negative integer and $\gamma$ is a curve in $\Omega$ that intersects both $\Om_m$ and $\Om_{m+2}$, then
\begin{equation}\label{eq:crossing-levels}
\ell_\pip(\gamma)\ge \frac{\pip(2^m)}{C_\pip^2}\ell_d(\gamma\cap\Om_{m+1}))\ge \frac{\pip(2^m)}{C_\pip^2}\, 2^m.
\end{equation}

\begin{lemma}\label{lem:nearby-points}
Let $x\in\Om_m$ for some integer $m\ge 0$. If $y\in\Om$ is such that 
\[
d_\pip(x,y)<\left[\min\left\{\tfrac{10}{11C_\pip^2}\, 2^{-2},\tfrac{10}{{22}C_q^2}\right\}\right]\, \pip(2^m)\, 2^m, 
\]
then 
\[
\pip(2^{m+1})\, d(x,y)\le \tfrac{11}{10}\, d_\pip(x,y)\le C_A\,\pip(2^{m})\, d(x,y).
\]
In particular, $C_A^{-1}\, \pip(2^m)\, d(x,y)\le d_\pip(x,y)\le C_A\pip(2^m)\, d(x,y)$.
\end{lemma}

Here the constant $C_\pip$ is the reverse doubling constant, preventing uncontrolled decay of the 
dampening function $\pip$.
It follows from the above lemma that the two metrics are locally bi-Lipschitz equivalent in $\Om$, and generate the same 
topology there. This lemma is applicable even when $\Om$ is not a uniform domain, but we need $\Om$ to be quasiconvex.

\begin{proof}
Let $\gamma$ be a curve in $\Om$ connecting $x$, $y$ such that 
$\ell_\pip(\gamma)\le \tfrac{11}{10}\, d_\pip(x,y)$. 
Then 
\[
\ell_\pip(\gamma)<\pip(2^m)\, 2^{m-2}/C_\pip^2,
\] 
and so by~\eqref{eq:crossing-levels}
we have that $\gamma$ does not intersect $\Om_{m+2}$. If $m\ge2$ and $\gamma$ intersects $\Om_{m-2}$, then 
by~\eqref{eq:crossing-levels} again, we would have
$\ell_\pip(\gamma)\ge \pip(2^{m-2})\, 2^{m-2}/C_\pip^2\ge \pip(2^m)\, 2^{m-2}/C_\pip^2$, 
{which again violates the above inequality.}
It follows that $\gamma\subset\Om_{m-1}\cup\Om_m\cup\Om_{m+1}$, where, for convenience, we set $\Om_{n}=\emptyset$
when $n$ is a negative integer. Hence
\begin{equation}\label{eq:first-inequality}
\tfrac{11}{10}d_\pip(x,y)\ge \ell_\pip(\gamma)\ge \pip(2^{m+1})\ell_d(\gamma)\ge \pip(2^{m+1})d(x,y),
\end{equation}
which proves the first inequality of the desired double inequality claimed in the lemma.

On the other hand, as $\Omega$ is a quasiconvex space, 
we can find a curve $\beta$ with end points $x,y$ such that $\ell_d(\beta)\le C_qd(x,y)$ where $C_q$ is the quasiconvexity 
constant of the metric $d$ on $\Om$. 
We consider two cases, $C_q<2C_\pip$ and $C_q\ge 2C_\pip$.

In the first case, $C_q< 2 C_\pip$. Here we use that $d_\pip(x,y)<\tfrac{10}{11C_\pip^2}\pip(2^m)\, 2^{m-2}$ 
and so~\eqref{eq:first-inequality} implies that $d(x,y)<C_\pip^{-1}\, 2^{m-2}$. Hence, for each $z$ in the trajectory of $\beta$,
\[
 d_\Om(z)\le d_\Om(x)+\ell_d(\beta) < 2^{m}+\frac{C_q}{C_\pip} 2^{m-2}=\left(1+\frac{C_q}{4C_\pip}\right)2^m<  2^{m+1}
\]
and 
\[
d_\Om(z)\ge d_\Om(x)-\ell_d(\beta) \ge 2^{m-1}-\frac{C_q}{C_\pip} 2^{m-2}=A_1\, 2^{m-1}.
\]
Here, $A_1=1-\frac{C_q}{2C_\pip}>0$.

In the second case $C_q\ge 2 C_\pip$. Here we use that $d_\pip(x,y)<\tfrac{10}{{22}C_q^2}\pip(2^m)\,2^{m}$, 
and so it follows from~\eqref{eq:first-inequality} that $d(x,y)<C_\pip C_q^{-2}\, 2^{{m-1}}$. 
Hence, for each $z$ in the trajectory of $\beta$,
\[
d_\Om(z)\le d_\Om(x)+\ell_d(\beta) \le 2^{m}+\frac{C_\pip}{C_q} 2^{m}=\left(1+\frac{C_\pip}{C_q}\right)2^m\leq 2^{m+1}
\]
and 
\[
d_\Om(z)\ge d_\Om(x)-\ell_d(\beta) \ge 2^{m-1}-\frac{C_\pip}{C_q} 2^{{m-1}}=A_2\, 2^{m-1}.
\]
Here, $A_2=1-\frac{C_\pip}{C_q}>0$. 

{Let $k_0$ be the positive integer such that $2^{-k_0}<A_1\le 2^{1-k_0}$ if $C_q<2C_\pip$ 
or the positive integer such that $2^{-k_0}<A_2\le 2^{1-k_0}$ if $C_q\ge 2C_\pip$.}
In either case, it 
follows that $\beta\subset \bigcup_{n=m-k_0}^{m+1}\Om_n$, and so
\[
d_\pip(x,y)\le \ell_\pip(\beta)\le \pip(2^{m-k_0-1})\ell_d(\beta)\le C_q\, C_\pip^{k_0+1}\pip(2^m) d(x,y).
\]
We set $C_A=C_q\, C_\pip^{k_0+1}$ to complete the proof.
\end{proof}

Note that if $x\in\Om_n$ and $d(x,y)<s\,2^n\pip(2^n)$ for sufficiently small $s>0$, 
then $d_\pip(x,y)$ satisfies the hypothesis of Lemma~\ref{lem:nearby-points}. Hence
the lemma also tells us that we should think of balls $B_d(x,s\, 2^n\pip(2^n))$ as Whitney-type balls in $d_\pip$;
note that the doubling property of $\pip$ guarantees that $\pip$ satisfies a Harnack-type condition on these balls, as
outlined in~\cite{BHK}.

\begin{lemma} \label{lem:dist-to-infty}
Let $x\in\Om_m$ for some integer $m\ge n_0+2$. Then 
\[
C_UC_\pip\sum_{n=m-n_0}^\infty2^n\pip(2^n)\ge d_\pip(x,\infty)\ge \frac{5}{11}\sum_{n=m+1}^\infty 2^n\pip(2^n).
\]
Here $n_0$ is the positive integer such that $2^{n_0-1}\le C_U<2^{n_0}$ with $C_U$ the uniformity constant
associated with the uniform domain $(\Om,d)$.
\end{lemma}

\begin{proof}
Let $(x_j)_{j>m}$ be a sequence of points in $\Om$ such that $x_j\in\Om_{j}$. Then this sequence is not convergent in $\Om$.
Let $\beta_j$ be a uniform curve in $\Om$ with end points $x,x_j$. Note that
$\ell_d(\beta_j)\le C_U d(x,x_j)$ where $C_U$ is the uniformity constant of $\Om$. 
Recall from Section~2 that we only consider $C_U$-uniform curves (with respect to the metric $d$) whose each subcurve is
also $C_U$-uniform, see~\cite{BHK} for more on this.
It follows from the local compactness of $\Omega$ and the Arzel\`{a}-Ascoli theorem that there is a locally uniformly convergent
subsequence of the sequence of curves $\beta_j$, that converges to a curve $\beta$ with one end point $x$ and leaving each 
compact subdomain of $\overline{\Om}$; note that each $\beta_j$ lies in $\Om\setminus\Om_0$, and hence so does $\beta$.
We also have $\lim_{t\to\infty}d(x,\beta(t))=\infty$, and so $\beta$ connects $x$ to $\infty$. 
For each positive integer $n$ let $\widehat{\beta_n}=\beta\cap\Om_n$. By the uniformity of each $\beta_j$ we know that
$\beta_j$ does not intersect $\Om_{m-n_0-1}$, and hence neither does $\beta$.
Then
\begin{align*}
d_\pip(x,\infty)\le \ell_\pip(\beta)=\sum_{n=m-n_0}^\infty\ell_\pip(\widehat{\beta_n})
  &\le \sum_{n=m-n_0}^\infty \pip(2^{n-1})\ell_d(\widehat{\beta_n})\\ 
  &\le \sum_{n=m-n_0}^\infty \pip(2^{n-1})C_Ud_\Om(z_n)\\
  &\le C_UC_\pip\sum_{n=m-n_0}^\infty \pip(2^n)\,2^n,
\end{align*}
where $z_n$ is a point in $\widehat{\beta_n}$. In particular, this also means that $d_\pip(x,\infty)$ is finite 
by~\eqref{eq:pip-reverse-doubling}.

Now let $\gamma:[0,\infty)\to\Om$ be any curve in $\Om$ such that $\gamma(0)=x$ and $\lim_{n\to\infty}d(x,\gamma(t))=\infty$
and $\ell_\pip(\gamma)\le \tfrac{11}{10}d_\pip(x,\infty)$. Letting $\gamma_n=\gamma\cap\Om_n$, we see that
\begin{align*}
\frac{11}{10}d_\pip(x,\infty)\ge \sum_{n=m+1}^\infty\ell_\pip(\gamma_n)
  \ge \sum_{n=m+1}^\infty \pip(2^n)\, \ell_d(\gamma_n)%\\
  \ge \sum_{n=m+1}^\infty \pip(2^n)\,2^{n-1}.
\end{align*}
\end{proof}

Thanks to the above lemma, we know that $\Om_\pip$ is a bounded domain. Indeed, when $x,y\in\Om$, we have that
\[
d_\pip(x,y)\le d_\pip(x,\infty)+d_\pip(y,\infty)\le 2C_UC_\pip\sum_{n=0}^\infty 2^n\pip(2^n).
\]

\begin{lemma}\label{lem:dist-pip-bdy}
Let $x\in\Om_m$ for some non-negative integer $m$. If $m>0$, then 
\[
C_UC_\pip\sum_{n=0}^{m+n_0}2^n\pip(2^n)\ge \dOmp(x)\ge \biggl(\frac{50}{121}\biggr)\,\sum_{n=0}^{m-1} 2^n\pip(2^n).
\]
If $m=0$, then $\dOmp(x)= d_\Om(x)$.
Here, as usual, $n_0$ is the positive integer such that $2^{n_0-1}\le C_U< 2^{n_0}$ with $C_U$ the uniformity constant
associated with the uniform domain $(\Om,d)$.
\end{lemma}

\begin{proof}
By Lemma~\ref{lem:d-bdry-to-phi} we know that $\partial\Om_\pip=\partial\Om$. 
Let $\zeta\in\partial\Om$ be such that
$d_\Om(x)=d(x,\zeta)$, and let $\beta$ be a $C_U$-uniform curve (with respect to the metric $d$)
with end points $x$ and $\zeta$. Then, 
\[
2^{m-1}\le d_\Om(x)\le \ell_d(\beta)\le C_U d(x,\zeta)\le C_U\, 2^{m}\le 2^{m+n_0}.
\]
It follows that the trajectory of $\beta$ lies in $\bigcup_{n=0}^{m+n_0}\Om_n$. Therefore,
\begin{align*}
\ell_\pip(\beta)\le\sum_{n=0}^{m+n_0}\pip(2^{n-1})\ell_d(\beta\cap\Om_n)
 &\le C_\pip\, \sum_{n=0}^{m+n_0}\pip(2^n)\, C_Ud_\Om(z_n)\\
 &\le C_UC_\pip\, \sum_{n=0}^{m+n_0}\pip(2^n)\, 2^{n}\\
 &= C_UC_\pip\, \sum_{n=0}^{m+n_0}\pip(2^n)\, 2^n,
\end{align*}
where $z_n$ is a point in $\beta\cap\Om_n$. Hence,
\[
\dOmp(x)\le \ell_\pip(\beta)\le C_UC_\pip\, \sum_{n=0}^{m+n_0}\pip(2^n)\, 2^n.
\]

If $m\ge 1$, then
let $\zeta\in\partial\Om$ such that $\dOmp(x)\ge \tfrac{10}{11}d_\pip(x,\zeta)$, and let $\gamma$ be a curve
in $\Om$ connecting $x$ to $\zeta$ such that $\ell_\pip(\gamma)\le \tfrac{11}{10}d_\pip(x,\zeta)$. Then
\begin{align*}
\ell_\pip(\gamma)\ge \sum_{n=0}^{m-1}\ell_\pip(\gamma\cap\Om_n)
  \ge \sum_{n=0}^{m-1} \pip(2^{n})\ell_d(\gamma\cap\Om_n)
  \ge \sum_{n=0}^{m-1}\pip(2^n) 2^{n-1}.
\end{align*}
It follows that
\[
\dOmp(x)\ge\frac{1}{2} \biggl(\frac{10}{11}\biggr)^{\!2}\, \sum_{n=0}^{m-1}\pip(2^n) 2^{n}.
\]
If $m=0$, then $\pip(t)=1$ for $0< t\le 1$ tells us that $\dOmp(x)\ge d_\Om(x)$.
\end{proof}

\section{Uniform domain property of $\Om_\pip$.}

This section is devoted to the proof of Theorem~\ref{thm:main}, the main theorem of this note.

\begin{lemma}\label{small}
Suppose that $x\in\Om_m$ and $y\in\Om_k$ with $0\le m \le k$. If
\[
d_\pip(x,y)<\tfrac{5}{44\, C_\pip^{n_0+1}C_q^2C_UC_A} 2^{m}\pip(2^m), 
\]
then any $C_U$-uniform
curve with respect to the original metric $d$ with end points $x,y$ is a 
$C_{1}^{\pip}$-uniform curve with respect to the metric $d_\pip$, where
\[
C_1^\pip=\max\left\{ C_AC_\pip^{n_0+1}C_U,\,  
{\tfrac{363\, C_U}{50}}\right\}.
\]
Here $C_A$ is the constant from Lemma~\ref{lem:nearby-points}, which depends only on $C_\pip$ and the 
quasiconvexity constant $C_q$. 
\end{lemma}

\begin{proof}
We first consider the case $m\ge 1$.
By Lemma~\ref{lem:nearby-points}, we have that
$\pip(2^m)\,d(x,y)\le C_A\, d_\pip(x,y)$. It follows that 
\[
d(x,y)\le \frac{C_A}{\pip(2^m)} d_\pip(x,y) \le \tfrac{5}{44\, C_\pip^{n_0+1}C_q^2C_U} 2^{m},
\]
and so
\[
2^{m-1}\le d_\Om(y)\le d(x,y)+d_\Om(x)\le \tfrac{5}{44\, C_\pip^{n_0+1}C_q^2C_U} 2^{m}+2^{m+1}\le C_*\, 2^m.
\]
It follows that if $\beta$ is a $C_U$-uniform curve with respect to $d$ with end points $x,y$ that 
$\beta\subset\bigcup_{n=m-n_0}^{m+k_0+n_0}\Om_n$, where
$n_0$ is  the positive integer such that $2^{n_0-1}\le C_U < 2^{n_0}$, see
Remark~\ref{rem:def-n0}, 
and $k_0$ is the positive integer such that $2^{k_0-1}\le C_*< 2^{k_0}$. Hence,
\begin{align*}
\ell_\pip(\beta)=\sum_{n=m-n_0}^{m+k_0+n_0}\int_{\beta\cap\Om_n}\pip(d_\Om(\beta(t)))\, dt
  &\le \sum_{n=m-n_0}^{m+k_0+n_0}\pip(2^n)\ell_d(\beta\cap\Om_n)\\
  &\le \pip(2^{m-n_0{-1}})\, \ell_d(\beta)\\
  &\le C_\pip^{n_0{+1}}\, \pip(2^m)\, C_U\, d(x,y).
\end{align*}
Using Lemma~\ref{lem:nearby-points} again, we conclude that
\[
d_\pip(x,y)\ge C_A^{-1} \pip(2^m)\, d(x,y)\ge \frac{1}{C_AC_\pip^{n_0{+1}} C_U}\ell_\pip(\beta);
\]
that is, $\beta$ is a quasiconvex curve with respect to the metric $d_\pip$.

Next, if $z$ is a point in the trajectory of $\beta$, then by Lemma~\ref{lem:dist-pip-bdy},
\begin{align*}
\dOmp(z)\ge \dOmp(x)-d_\pip(x,z)
     &\ge \frac{50}{121}\sum_{n=0}^{m-1}2^n\pip(2^n)-\ell_\pip(\beta)\\
     &\ge \frac{50}{121}\, 2^{m-1}\pip(2^{m-1})-C_AC_\pip^{n_0+1}C_Ud_\pip(x,y)\\
     &\ge \frac{50}{121}\, 2^{m-1}\pip(2^{m-1})-\frac{5}{44C_q^2}2^m\pip(2^m)\\
     &\ge \frac{50}{121}\, 2^{m-1}\pip(2^{m-1})-\frac{5}{22C_q^2}2^{m-1}\pip(2^{m-1})\\
     &=\frac{45}{242}2^{m-1}\pip(2^{m-1}).
\end{align*}
As
\[
\ell_\pip(\beta)\le C_AC_\pip^{n_0+1}C_U\, d_\pip(x,y)\le \frac{5}{44C_q^2}\, 2^{m}\pip(2^{m}) 
\le \frac{5}{22C_q^2}\, 2^{m-1}\pip(2^{m-1}),
\]
it follows that $\beta$ is a $C_{1}^{\pip}$-uniform curve with respect to the metric $d_\pip$. % where

Now we consider the case $m=0$; that is, $x\in\Om_0$. Then, by the assumption on $y$, we must have that $y\in\Om_0\cup\Om_1$.
If not, then any curve in $\Om$ that connects $x$ to $y$ must have a segment in $\Om_1$ with length at least $2$, and therefore
the $d_\pip$-length of all such curves are at least $2\, \pip(2)\ge 2/C_\pip$ which is larger than the assumed bound on $d_\pip(x,y)$.
Moreover, by Lemma~\ref{lem:nearby-points} we have that $d(x,y)\le 1/(4C_\pip)$.
Hence any $C_U$-uniform curve (in the metric $d$) with end points $x,y$ must lie in $\bigcup_{n=0}^{n_0+1}\Om_n$. Let 
$\beta$ be such a curve. 
We have that $\ell_\pip(\beta)\le \ell_d(\beta)$.
This implies, by Lemma~\ref{lem:nearby-points}, that 
\[
d_\pip(x,y)\geq C_A^{-1}d(x,y)\geq C_A^{-1}C_U^{-1}\ell_d(\beta) \geq C_A^{-1}C_U^{-1}\ell_\pip(\beta), 
\]
meaning that $\beta$ is quasiconvex with respect to the metric $d_\pip$. 

For $z$ in the trajectory of $\beta$, consider the segment $\beta[x,z]$ of $\beta$ with end points $x,z$. As
we require that subcurves of chosen uniform curves (with respect to the metric $d$)
also be uniform, $\beta$ has no loops, and so there is only one such segment.
If $z\in\Om_0$, then 
\[
\dOmp(z)=d_\Om(z)\geq C_U^{-1}\ell_d(\beta[x,z])\geq C_U^{-1}\ell_\pip(\beta[x,z]).
\]
If $z\in\Om_j$ for some $0<j\leq n_0+1$, then
\begin{align*}
\ell_\pip(\beta[x,z])=\sum_{n=0}^{j}\int_{\beta[x,z]\cap\Om_n}\pip(d_\Om(\beta(t)))\, dt
     &\le \sum_{n=0}^{j}\pip(2^n)\ell_d(\beta[x,z]\cap\Om_n)\\
     &\le \sum_{n=0}^{j}\pip(2^n)\ell_d\left(\bigcup_{i=1}^{n}\beta[x,z]\cap\Om_i\right)\\
     &\le C_U\sum_{n=0}^{j}\pip(2^n)2^n.
\end{align*}
Thus, noting that $2^{j-1}\pip(2^{j-1})+2^j\pip(2^j)\leq 3(2^{j-1}\pip(2^{j-1}))$, from Lemma~\ref{lem:dist-pip-bdy} it follows that 
\begin{align*}
\dOmp(z)\geq \frac{50}{121}\sum_{n=0}^{j-1}2^n\pip(2^n)
     &= \frac{50}{121}\left[\sum_{n=0}^{j-2}2^n\pip(2^n)+2^{j-1}\pip(2^{j-1})\right]\\
     &\ge\frac{50}{363}\sum_{n=0}^{j}2^n\pip(2^n)\\
     &\ge \frac{50}{363\,C_U}\ell_\pip(\beta_{x,z}).
\end{align*}
This shows that $\beta$ is a $C_1^\pip$-uniform curve with respect to $d_\pip$. 
\end{proof}

From equation~\eqref{eq:pip-reverse-doubling} it follows that we can fix a positive integer $m_0>n_0+2$ such that 
\begin{equation}\label{eq:m0-cond}
\sum_{n=m_0-n_0}^\infty 2^n\pip(2^n)<\frac{1}{8C_UC_\pip}.
\end{equation}

\begin{lemma}\label{lem:large-k}
Suppose that $x\in\Om_m$ and $y\in\Om_k$ with $m_0\le m \le k$. If $\gamma$ is a curve in $\Om$ with end points
$x$ and $y$ such that $\ell_\pip(\gamma)\le \tfrac{11}{10}d_\pip(x,y)$, then $\gamma$ is a $\tfrac{1331}{669}$-uniform
curve with respect to the metric $d_\pip$.
\end{lemma}

\begin{proof}
Suppose that $x\in\Om_m$ and $y\in\Om_k$ with $k\ge m\ge m_0$. Then by Lemma~\ref{lem:dist-to-infty},
$d_\pip(x,y)\le 2C_UC_\pip\sum_{n=m-n_0}^\infty 2^n\pip(2^n)<\tfrac{1}{4}$, and moreover, by Lemma~\ref{lem:dist-pip-bdy} we 
also have
\[
d_{\Om_\pip}(x)\ge \frac{50}{121}\sum_{n=0}^{m-1}2^n\pip(2^n)\ge \frac{50}{121}.
\]
Similar statement holds also for $d_{\Om_\pip}(y)$. Let $\gamma$ be a curve in $\Om$ with end points $x,y$ such that 
$\ell_\pip(\gamma)\le \tfrac{11}{10}d_\pip(x,y)$. Then $\ell_\pip(\gamma)<\tfrac{11}{40}$. Let $z$ be a point in the trajectory
of $\gamma$; then, 
\[
\dOmp(z)\ge \dOmp(x)-d_\pip(x,z)\ge \frac{50}{121}-\ell_\pip(\gamma)\ge \frac{50}{121}-\frac{11}{40}
  =\frac{669}{4840}.
\]
It follows that
\[
\dOmp(z)\ge \frac{669}{1331}\, \ell_\pip(\gamma),
\]
that is, $\gamma$ is a $\tfrac{1331}{669}$-uniform curve with respect to the metric $d_\pip$.
\end{proof}

In what follows, we denote by $\lambda$ and $\Lambda$ the numbers 
\begin{equation}\label{eq:lam-Lam}
\lambda=\min_{0\leq n \leq m_0+n_0}2^n\pip(2^n)\quad\text{and}\quad\Lambda= \max_{0\leq n \leq m_0+n_0}2^n\pip(2^n).
\end{equation}

\begin{lemma}\label{medium}
Suppose that $x\in\Om_m$ and $y\in\Om_k$ with $0\leq m\leq k \le m_0$, and 
\[
\frac{5}{22C_\pip^2}2^m\pip(2^m)\leq d_\pip(x,y)<C\,2^{m}\pip(2^m). 
\] 
Any $C_U$-uniform
curve with respect to the original metric $d$ with end points $x,y$ lying entirely in $\bigcup_{j=0}^{m_0+n_0}\Om_j$ is a 
$C_{2}^{\pip}$-uniform curve with respect to the metric $d_\pip$.

If the uniform curve is not entirely contained in $\bigcup_{j=0}^{m_0+n_0}\Om_j$, then
with $z_1,z_2$ two points in the trajectory of the curve with the segment between $x$ and $z_1$,
and the segment between $z_2$ and $y$ lying in $\bigcup_{j=0}^{m_0+n_0}\Om_j$, we can replace
the segment between $z_1$ and $z_2$ by a $11/10$-quasiconvex curve with respect to $d_\pip$
with end points $z_1,z_2$ to obtain a $C_2^\pip$-uniform curve with respect to the metric $d_\pip$.

Here
\[
C_2^\pip=\frac{2000}{669}\, \frac{2CC_1^\pip}{T}\, \frac{\Lambda}{\lambda} \, \left(2T_0+\frac{121 C_\pip^2}{20 \lambda}\right).
\]
Moreover, in both cases, for each point $z$ in the trajectory of $\beta$ (resp. $\gamma$), we have that
$C_2^\pip\, d_{\Om_\pip}(z)$ is minorized by the length of the entire curve with respect to the metric $d_\pip$.
\end{lemma}

\begin{proof}
Let $\beta$ be a $C_U$-uniform curve (with respect to $d$) with end points $x,y$ with arclength (with respect to $d$) parametrization $\beta:[0,L]\rightarrow\Omega$. We can find $t_1,t_2,\ldots, t_{J-1}\in(0,L)$ such that 
\[
0=t_0<t_1<t_2<\cdots<t_{J-1}<t_J=L
\]
and for $j=1,\ldots,J$,
\[
d_\pip(\beta(t_j),\beta(t_{j-1}))<\tfrac{5}{22\, C_\pip^{n_0+1}C_UC_A} 2^{m_j}\pip(2^{m_j})= T2^{m_j}\pip(2^{m_j})
\]
with 
\[
d_\pip(\beta(t_j),\beta(t_{j-1}))\ge T2^{m_j-1}\pip(2^{m_j}). 
\]
Here $m_j$ is chosen such that either $\beta(t_j)\in\Om_{m_j}$ or $\beta(t_{j-1})\in\Om_{m_j}$. 
With $z_j=\beta(t_j)$, note by the hypotheses of the lemma that when $m_j\le m_0+n_0$,
\begin{align*}
d_\pip(z_j,z_{j-1})<T 2^{m_j}\pip(2^{m_j})&\le T\Lambda \frac{22C_\pip^2}{5\,\cdot\, 2^m\pip(2^m)}d_\pip(x,y)\\
 &\le \frac{22C_\pip^2}{5}\, T\frac{\Lambda}{\lambda}\, d_\pip(x,y).
\end{align*}
The remaining proof is split into two cases.

\noindent{\bf Case~1:} $\beta\subset\bigcup_{j=0}^{m_0+n_0}\Om_j$. It follows that $0\leq m_j\leq m_0+n_0$.
Then
by the $C_U$-uniformity of $\beta$ with respect to the metric $d$, we have that for the midpoint $z\in\beta$,
\[
2^{m_0+n_0}\ge d_\Om(z)\ge \frac{1}{C_U}\frac{\ell_d(\beta)}{2}\ge \frac{1}{2C_U} d(x,y),
\]
that is, $d(x,y)\le 2^{m_0+n_0+1}C_U$. Hence 
\[
J\leq \frac{\ell_d(\beta)}{T\lambda}
\le \frac{C_U d(x,y)}{T\lambda} \le \frac{2^{m_0+n_0+1}C_U^2}{T\lambda}.
\]
Applying Lemma~\ref{small} to each subcurve $\beta_j$ connecting $z_j=\beta(t_j)$ and $z_{j-1}=\beta(t_{j-1})$, we have that
\begin{align*}
\ell_\pip(\beta)= \sum_{j=1}^{J}\ell_\pip(\beta_j) \leq C_1^\pip\sum_{j=1}^{J}d_\pip(z_j,z_{j-1})
&\leq C_1^\pip J \frac{22C_\pip^2}{5}T\frac{\Lambda}{\lambda} d_\pip(x,y)\\
&\leq  C_1^\pip \frac{22C_\pip^2}{5}T\frac{\Lambda}{\lambda}\frac{2^{m_0+n_0+1}C_U^2}{T\lambda}\,
d_\pip(x,y).
\end{align*}
We set
\[
T_0:=C_1^\pip \frac{22C_\pip^2}{5}\frac{\Lambda}{\lambda}\frac{2^{m_0+n_0+1}C_U^2}{\lambda}.
\]
Moreover, any $z$ in the trajectory of $\beta$ is in the trajectory of $\beta_j$ for some $j$ and so, applying 
Lemma~\ref{small} to this curve, 
\begin{align*}
\dOmp(z)\ge \frac{1}{C_1^\pip}\ell_\pip(\beta_j)\ge \frac{1}{C_1^\pip}d_\pip(z_j,z_{j-1})
     &\ge \frac{1}{C_1^\pip}\frac{T}{2}\lambda\frac{1}{C\, 2^m\pip(2^m)}d_\pip(x,y)\\
     &\ge \frac{T}{2C\, C_1^\pip}\frac{\lambda}{\Lambda}\, d_\pip(x,y)\\
     &\ge \frac{T}{2C\, C_1^\pip}\frac{\lambda}{\Lambda}\,\frac{1}{T_0}\, \ell_\pip(\beta).
\end{align*}

\noindent{\bf Case 2:} There is some $z$ in the trajectory of $\beta$ such that $d_\Om(z)>2^{m_0+n_0}$. Let $z_1,z_2$ be two
points in the trajectory of $\beta$ such that $d_\Om(z_1)=d_\Om(z_2)=2^{m_0+n_0}$ and $\beta[x,z_1]$, $\beta[z_2,y]$
lie entirely in $\bigcup_{j=0}^{m_0+n_0}\overline{\Om_j}$. In this case, we replace $\beta[z_1,z_2]$ with a curve $\widehat{\beta}$
with end points $z_1,z_2$ such that $\ell_\pip(\widehat{\beta})\le \tfrac{11}{10}d_\pip(z_1,z_2)$. By Lemma~\ref{lem:dist-to-infty}
and by~\eqref{eq:m0-cond}, we have that
\[
\ell_\pip(\widehat{\beta})\le \frac{11}{10} \left[ d_\pip(z_1,\infty)+d_\pip(z_2,\infty)\right]\le \frac{11}{40}
  \le \frac{11}{40} \frac{22C_\pip^2}{\lambda} d_\pip(x,y).
\]
Considering the subdivisions of $\beta[x,z_1]$ and $\beta[z_2,y]$ as before, we get
\[
\ell_\pip(\beta[x,z_1])\le \sum_{j=1}^{J_1}d_\pip(z_j,z_{j-1})\le T_0\, d_\pip(x,y)
\]
and 
\[
\ell_\pip(\beta[z_2,y])\le \sum_{j=J_2}^J d_\pip(z_j,z_{j-1})\le T_0\, d_\pip(x,y).
\]
Here we used the fact that both $J_1$ and $J-J_2$ satisfy the estimates given in Case~1 for $J$.
Thus, with $\gamma$ the concatenation of the three curves $\beta[x,z_1]$, $\widehat{\beta}$, and $\beta[z_2,y]$,
we obtain
\[
\ell_\pip(\gamma)\le \left(2T_0+\frac{121 C_\pip^2}{20 \lambda}\right)d_\pip(x,y).
\]

Let $z\in\gamma$. If $z\in\beta[x,z_1]$ or if $z\in\beta[z_2,y]$, then as in Case~1 above, we obtain
\[
d_{\Om_\pip}(z)\ge \frac{T}{2C\, C_1^\pip}\frac{\lambda}{\Lambda} d_\pip(x,y)
  \ge \frac{T}{2C\, C_1^\pip}\frac{\lambda}{\Lambda}\,\left(2T_0+\frac{121 C_\pip^2}{20 \lambda}\right)^{-1}\ell_\pip(\gamma).
\]
If $z\in\widehat{\beta}$, then by Lemma~\ref{lem:large-k},
\[
d_{\Om_\pip}(z)\ge \frac{669}{1331}\ell_\pip(\widehat{\beta}).
\]
Hence
\[
d_{\Om_\pip}(z)\ge d_{\Om_\pip}(z_1)-d_\pip(z,z_1)\ge d_{\Om_\pip}(z_1)-\ell_\pip(\widehat{\beta})
\]
and so, by the inequality above, we have
\[
d_{\Om_\pip}(z)\ge \frac{669}{2000}d_{\Om_\pip}(z_1)
\ge \frac{669}{2000}\, \frac{T}{2C\, C_1^\pip}\frac{\lambda}{\Lambda}\,\left(2T_0+\frac{121 C_\pip^2}{20 \lambda}\right)^{-1}\ell_\pip(\gamma).
\]
\end{proof}

Recall the definition of $\lambda$ and $\Lambda$ from~\eqref{eq:lam-Lam} above.

\begin{lemma}\label{lem:large-bound}
Suppose that $x\in\Om_m$ and $y\in\Om_k$ with $0\leq m\leq k \le m_0$. Then,
\[
d_\pip(x,y)\leq \frac{2^{m_0+n_0+1}C_U^2}{\lambda}\,2^{m}\pip(2^m). 
\] 
\end{lemma}

\begin{proof}
Suppose that $x,y$ are as in the hypothesis of the lemma, and that
\[
d_\pip(x,y)> \frac{2^{m_0+n_0+1}C_U^2}{\lambda}\,2^{m}\pip(2^m). 
\] 
Let $\beta$ be a $C_U$-uniform curve (with respect to $d$) with end points $x,y$. 
By the {above supposition},
there is some point $z$ in the trajectory of $\beta$ such that $d_{\Omega}(z)> 2^{m_0+n_0}$. Let $z_1, z_2$ be 
two points in the trajectory of $\beta$ such that $d_\Om(z_1)=d_\Om(z_2)=2^{m_0}$ and $\beta[x,z_1]$, 
$\beta[z_2,y]$ lie entirely in $\bigcup_{j=0}^{m_0}\overline{\Om_j}$. We replace $\beta[z_1,z_2]$ with a curve $\widehat{\beta}$
with end points $z_1,z_2$ such that $\ell_\pip(\widehat{\beta})\le \tfrac{11}{10}d_\pip(z_1,z_2)$. By 
{the supposition assumed at the beginning of the proof again}, with $C=\frac{2^{m_0+n_0+1}C_U^2}{\lambda}$, we have 
\begin{align*}
\ell_\pip(\beta[x,z_1])
\le\ell_d(\beta[x,z_1])
\le C_U d_\Omega(z_1) 
{=}  C_U 2^{m_0}
\le \frac{ C_U 2^{m_0}}{\lambda C} d_\pip(x,y). 
\end{align*}
Similarly, we get
\[
\ell_\pip(\beta[z_2,y])\le \frac{ C_U 2^{m_0}}{\lambda C} d_\pip(x,y).
\]
Moreover, by Lemma~\ref{lem:dist-to-infty} and~\eqref{eq:m0-cond},
\[
\ell_\pip(\widehat{\beta})\le \frac{11}{10}d_\pip(z_1,z_2)\le\frac{11}{40}
\le\frac{11}{40}\ \frac{1}{C\lambda} d_\pip(x,y).
\]
It follows that
\[
\ell_\pip(\gamma)\le \left(\frac{2^{m_0+1}C_U}{\lambda C}
+\frac{11}{40C\,\lambda}\right)\, d_\pip(x,y)<d_\pip(x,y),
\]
which is not possible.
\end{proof}

\begin{lemma}\label{lem:cross-border}
Suppose that $x\in\Om_m$ and $y\in\Om_k$ with $0\le m<m_0<k$. 
If 
\[
d_\pip(x,y)\ge \frac{5\lambda}{44C_\pip^{n_0+1}C_q^2C_UC_A}\, 
\]
then with $\beta$ a $C_U$-uniform curve (with respect to the metric $d$) with end points $x,y$ and with $z_1$ a point in
the trajectory of $\beta$ such that $d_\Om(z_1)=2^{m_0}$ and $\beta[x,z_1]$ contained in $\bigcup_{j=0}^{m_0}\overline{\Om_j}$,
and $\widehat{\beta}$ a curve with end points $z_1$ and $y$ such that $\ell_\pip(\widehat{\beta})\le \tfrac{11}{10}d_\pip(z_1,y)$,
the concatenation $\gamma$ of $\beta[x,z_1]$ and $\widehat{\beta}$ is a $C_3^\pip$-uniform curve
with respect to the metric $d_\pip$ with end points $x,y$. 

Here $C_3^\pip$ is the larger of the two following numbers:
\begin{align*}
C_2^\pip&+\frac{44C_\pip^{n_0+1}C_q^2C_UC_A}{5\lambda}\left(\frac{C_2^\pip}{2C_\pip}+\frac{11}{40}\right),\\
C_2^\pip&\left[1+\left(\frac{11}{40C_2^\pip}+\frac{1}{2C_\pip}\right)\frac{2000\,C_2^\pip}{669\lambda}\right].
\end{align*}
\end{lemma}

\begin{proof}
Let $\beta$, $z_1$, $\widehat{\beta}$, and $\gamma$ be as in the statement of the lemma. Then by  
Lemma~\ref{lem:dist-to-infty}, equation~\eqref{eq:m0-cond}, and Lemma~\ref{medium},
\begin{align*}
\ell_\pip(\gamma)=\ell_\pip(\beta[x,z_1])+\ell_\pip(\widehat{\beta})
   &\le C_2^\pip\, d_\pip(x,z_1)+\frac{11}{40}\\
   &\le C_2^\pip \left[d_\pip(x,y)+d_\pip(z_1,y)\right]+\frac{11}{40}\\
   &\le C_2^\pip\, d_\pip(x,y)+\left(\frac{C_2^\pip}{2C_\pip}+\frac{11}{40}\right)\\
    &\le \left[C_2^\pip +\left(\frac{C_2^\pip}{2C_\pip}
+\frac{11}{40}\right)\frac{44C_\pip^{n_0+1}C_q^2C_UC_A}{5\lambda}\right]\, d_\pip(x,y),
\end{align*}
showing that $\gamma$ is quasiconvex with respect to $d_\pip$. 

Now, if $z$ is a point in the trajectory of $\beta[x,z_1]$, then by Lemma~\ref{medium} we have that
\[
d_{\Om_\pip}(z)\ge \frac{1}{C_2^\pip}\ell_\pip(\beta[x,z_1]).
\]
If $z$ is a point in $\widehat{\beta}$, then 
\begin{align*}
d_{\Om_\pip}(z)\ge d_{\Om_\pip}(z_1)-d_\pip(z,z_1)
   &\ge \frac{1}{C_2^\pip}\ell_\pip(\beta[x,z_1])-\frac{1}{2C_\pip}\\
   &\ge \frac{1}{C_2^\pip}\ell_\pip(\gamma)-\left(\frac{\ell_\pip(\widehat{\beta})}{C_2^\pip}+\frac{1}{2C_\pip}\right)\\
   &\ge \frac{1}{C_2^\pip}\ell_\pip(\gamma)-\left(\frac{11}{40 C_2^\pip}+\frac{1}{2C_\pip}\right).
\end{align*}
Also, by Lemma~\ref{lem:large-k},
\[
d_{\Om_\pip}(z)\ge \frac{669}{1331}\ell_\pip(\widehat{\beta}),
\]
and so 
\[
d_{\Om_\pip}(z)\ge \frac{1}{C_2^\pip}\ell_\pip(\beta[x,z_1])-\ell_\pip(\widehat{\beta})
  \ge \frac{2^{m_0}\pip(2^{m_0})}{C_2^\pip}-\frac{1331}{661}d_{\Om_\pip}(z),
\]
from whence we obtain
\[
\frac{2000}{669}\, d_{\Om_\pip}(z)\ge \frac{\lambda}{C_2^\pip}.
\]
Thus, we finally get
\[
\frac{1}{C_2^\pip}\ell_\pip(\gamma)\le \left[1+\left(\frac{11}{40C_2^\pip}+\frac{1}{2C_\pip}\right)\frac{2000\,C_2^\pip}{669\lambda}\right]\, d_{\Om_\pip}(z),
\]
implying that $\gamma$ is a $C_3^\pip$-uniform curve with respect to $d_\pip$. 
\end{proof}

\begin{lemma}\label{lem:to-infinity-and-beyond}
Suppose that $x\in\Om$ and $y=\infty$. Then, there exists a $C_4^\pip$-uniform curve with respect to $d_\pip$ with 
end points $x,y$. Here
\[
C_4^\pip=\max\left\{\frac{1331}{669}, C_3^\pip\right\}.
\]
\end{lemma}

\begin{proof}
Let $x\in\Om_m$ for some non-negative integer $m$. 
If $m\ge m_0$, then as in the proof of Lemma~\ref{lem:dist-to-infty} we can find a curve $\beta$ beginning from $x$ 
and with $\lim_{t\to\infty}\beta(t)=\infty$, such that $\ell_\pip(\beta)\le \tfrac{11}{10} d_\pip(x,\infty)<\frac{11}{80}$.
Here $\beta:[0,\infty)\to\Om$. By considering $x$, $\beta(t)$, and $\beta\vert_{[0,t]}$ in Lemma~\ref{lem:large-k},
we see that $\beta[0,t]$ is a $\tfrac{1331}{669}$-uniform curve with respect to the metric $d_\pip$ for each $t>0$.
It follows that $\beta$ is a $\tfrac{1331}{669}$-uniform curve with respect to $d_\pip$ as well.

Now we consider the case that $m<m_0$. Let $\beta$ be a $C_U$-uniform curve (with respect to the metric $d$)
as constructed in Lemma~\ref{lem:noncomplete} such that $\beta:[0,\infty)\to\Om$ with $\lim_{t\to\infty}\beta(t)=\infty$.
We fix $k\ge m_0+n_0$ such that for each $z\in\Om_k$ we have 
$d_\pip(z,x)\ge \tfrac{5\lambda}{44C_\pip^{n_0+1}C_q^2C_UC_A}$ as in Lemma~\ref{lem:cross-border}. If no such 
$k$ exists, then we can directly apply Lemma~\ref{lem:cross-border} to $\beta$ to see that $\beta$ is a $C_3^\pip$-uniform
curve. With the choice of such $k$, let $\tau=\inf\{t>0:\, \beta(t)\in\bigcup_{j=k}^\infty\Om_j\}$, and we 
set $\gamma$ to be the concatenation of $\beta\vert_{[0,\tau]}$ with a curve $\widehat{\beta}$ with end points
$\beta(\tau)$ and $\infty$ such that $\ell_\pip(\widehat{\beta})\le \tfrac{11}{10}d_\pip(\beta(\tau),\infty)$.
An application of Lemma~\ref{lem:cross-border} now tells us that $\gamma$ is a $C_3^\pip$-uniform curve with
respect to the metric $d_\pip$.

By combining the above two cases, we obtain a $C_4^\pip$-uniform curve with respect to the metric $d_\pip$ and
connecting $x$ to $\infty$; here
\[
C_4^\pip=\max\left\{\tfrac{1331}{669}, C_3^\pip\right\}.
\]
\end{proof}

Now we are ready to prove the main theorem of this note.

\begin{proof}[Proof of Theorem~\ref{thm:main}]
The second claim of the theorem was established in Section~2, and so we now focus on proving that 
$\Om_\pip$ is a uniform domain. To this end, let $x,y\in\Om_\pip$ with $x\ne y$.
If $x=\infty$ or $y=\infty$, then by Lemma~\ref{lem:to-infinity-and-beyond} we have a
$C_4^\pip$-uniform curve with respect to $d_\pip$ connecting $x$ to $y$. So it only remains to consider when 
$x,y\in\Om_\pip\setminus\{\infty\}=\Om$.

Let $m,k$ be two non-negative integers such that $x\in\Om_m$ and $y\in\Om_k$. Without loss of
generality, we assume that $m\le k$. 

With $n_0$ and $m_0$ positive integers such that $2^{n_0-1}\le C_U< 2^{n_0}$ and $m_0\ge n_0+2$ 
with $\sum_{n=m_0-n_0}^\infty 2^n\pip(2^n)<(8C_UC_\pip)^{-1}$ as in~\eqref{eq:m0-cond}, we consider
three cases.
\begin{enumerate}
\item $m_0\le m\le k$. In this case, by Lemma~\ref{lem:large-k} we have a $\tfrac{1331}{669}$-uniform curve
 with respect to $d_\pip$ connecting $x$ to $y$.
\item $0\le m\le k\le m_0$. In this case, Lemma~\ref{lem:large-bound} we know that
$d_\pip(x,y)\le 2^{m_0+n_0+1}C_U^2\lambda^{-1}\, 2^m\pip(2^m)$.
Hence, by Lemma~\ref{small} and by Lemma~\ref{medium} (with $C=2^{m_0+n_0+1}C_U^2\lambda^{-1}$),
there is a $\max\{C_1^\pip, C_2^\pip\}$-uniform curve, with respect to the metric $d_\pip$, connecting $x$
to $y$. 
\item $0\le m<m_0<k$. Then by Lemma~\ref{small} and Lemma~\ref{lem:cross-border}
there is a $\max\{C_1^\pip, C_3^\pip\}$-uniform curve with respect to the metric $d_\pip$ with end
points $x$ and $y$. 
\end{enumerate}

Since the above cases exhaust all the possibilities of $x,y\in\Om$, it follows that 
$\Om_\pip$ is $A_\pip$-uniform with respect to the metric $d_\pip$, with
\[
A_\pip=\max\{C_1^\pip, C_2^\pip, C_3^\pip, C_4^\pip, \tfrac{1331}{669}\}.
\]
\end{proof}

\noindent Address:\\

\noindent R.G.: Department of Mathematical Sciences, P.O. Box 210025, University of
Cincinnati, Cincinnati, OH 45221--0025, U.S.A. \\
\noindent E-mail: {\tt ryan.gibara@gmail.com}\\

\noindent N.S.: Department of Mathematical Sciences, P.O. Box 210025, University of
Cincinnati, Cincinnati, OH 45221--0025, U.S.A. \\
\noindent E-mail: {\tt shanmun@uc.edu} 


\begin{thebibliography}{99}

\bibitem{Azz} J. Azzam,
\emph{Sets of absolute continuity for harmonic measure in NTA domains},
Potential Anal. {\bf 45} (2016), no. 3, 403--433.

\bibitem{BBL} A. Bj\"orn, J. Bj\"orn, X. Li,
\emph{Sphericalization and $p$-harmonic functions on unbounded domains in Ahlfors regular spaces},
J. Math. Anal. Appl. {\bf 474} (2019), no. 2, 852--875.

\bibitem{BS} J. Bj\"orn, N. Shanmugalingam,
\emph{Poincaré inequalities, uniform domains and extension properties for Newton-Sobolev functions in metric spaces},
J. Math. Anal. Appl. {\bf 332} (2007), no. 1, 190--208.

\bibitem{BHK}M. Bonk, J. Heinonen, and P. Koskela,
\emph{Uniformizing Gromov hyerbolic spaces},
Ast\'erisque {\bf 270} (2001), vi+99.

\bibitem{BK} M. Bonk, B. Kleiner,
\emph{Rigidity for quasi-Möbius group actions},
J. Differential Geom. {\bf 61} (2002), no. 1, 81--106.

\bibitem{BHX} S. M. Buckley, D. A. Herron, X. Xie,
\textit{Metric space inversions, quasihyperbolic distance, and uniform spaces},
Indiana Univ. Math. J. {\bf 57} No. 2 (2008), 837--890.

\bibitem{CKKSS} L. Capogna, J. Kline, R. Korte, N. Shanmugalingam, M. Snipes,
\emph{Neumann problems for $p$-harmonic functions, and induced nonlocal opertors in metric measure spaces},
in preparation.

\bibitem{DL1} E. Durand-Cartagena, X. Li,
\emph{Preservation of bounded geometry under sphericalization and flattening: quasiconvexity and $\infty$-Poincar\'e inequality},
Ann. Acad. Sci. Fenn. Math. {\bf 42} (2017), no. 1, 303--324.

\bibitem{DL2} E. Durand-Cartagena, X. Li,
\emph{Preservation of $p$-Poincar\'e inequality for large p under sphericalization and flattening},
Illinois J. Math. {\bf 59} (2015), no. 4, 1043--1069.

\bibitem{HKST}J. Heinonen, P. Koskela, N. Shanmugalingam, and J. Tyson,
\textit{Sobolev spaces on metric measure spaces: an
approach based on upper gradients}, New Mathematical Monographs {\bf 27},
Cambridge University Press (2015), i--xi+448.


\bibitem{HSX}D. Herron, N. Shanmugalingam, X. Xie,
\textit{Uniformity from Gromov hyperbolicity},
Illinois J. Math. {\bf 52} No. 4 (2008), 1065--1109.

\bibitem{L} X. Li,
\emph{Preservation of bounded geometry under transformations of metric spaces},
Thesis (Ph.D.)–University of Cincinnati (2015), ProQuest LLC, 140 pp.

\bibitem{LS} X. Li, N. Shanmugalingam,
\emph{Preservation of bounded geometry under sphericalization and flattening},
Indiana Univ. Math. J. {\bf 64} (2015), no. 5, 1303--1341.

\bibitem{Lierl} J. Lierl,
\emph{Scale-invariant boundary Harnack principle on inner uniform domains in fractal-type spaces},
Potential Anal. {\bf 43} (2015), no. 4, 717--747.

\bibitem{Maly} L. Maly,
\emph{Trace and extension theorems for Sobolev-type functions in metric spaces},
preprint (2017), {\tt https://arxiv.org/abs/1704.06344}

\bibitem{MartSarv}  O. Martio, J. Sarvas,
\emph{ Injectivity theorems in plane and space}, 
Ann. Acad. Sci. Fenn. Ser. A I Math. {\bf 4 }(1979), no. 2, 383--401.

\bibitem{ZLL} Q. Zhou, Y. Li, X. Li,
\emph{Sphericalizations and applications in Gromov hyperbolic spaces},
J. Math. Anal. Appl. {\bf 509} (2022), no. 1, Paper No. 125948.

\end{thebibliography}
\end{document}